\documentclass{amsart} 

\usepackage{amsthm, amssymb, amsmath, amscd}
\usepackage{eurosym}
\usepackage{amsfonts}
\usepackage{amsmath}
\usepackage{setspace}
\usepackage{graphicx}
\usepackage{slashed}

\usepackage[all,cmtip]{xy}

\usepackage{indentfirst}

\newtheorem{theorem}{Theorem}[section]

\newtheorem{cor}[theorem]{Corollary}

\newtheorem{lemma}[theorem]{Lemma}

\newtheorem{prop}[theorem]{Proposition}

\theoremstyle{definition}
\newtheorem{define}[theorem]{Definition}

\newtheorem{remark}[theorem]{Remark}

\newcommand{\field}[1]{\mathbb{#1}}

\newcommand{\rz}{\field{R}/\field{Z}}

\begin{document}
\title[Index maps and mapping cones]{Analytic and topological index maps with values in the $K$-theory of mapping cones}
\author{Robin J. Deeley}
\date{\today}
\begin{abstract}
Index maps taking values in the $K$-theory of a mapping cone are defined and discussed.  The resulting index theorem can be viewed in analogy with the Freed-Melrose index theorem.  The framework of geometric $K$-homology is used in a fundamental way.  In particular, an explicit isomorphism from a geometric model for $K$-homology with coefficients in a mapping cone, $C_{\phi}$, to $KK(C(X),C_{\phi})$ is constructed.
\end{abstract}
\maketitle

\section{Introduction}
The Baum-Douglas or $(M,E,f)$ model for $K$-homology is a fundamental tool in the study of index theory.  Since its introduction in \cite{BD}, it has been used to study both classical and exotic index theory. In particular, it is useful to construct variants of the Baum-Douglas model which are associated to various index problems; for example, models associated to non-integer valued index maps are of interest. We refer to the Baum-Douglas model and its variants as geometric models and assume the reader is familiar with the original $(M,E,f)$-model, see any of \cite{BD, BD2, BHS, EM, Rav, Wal}.
\par
This paper is continuation of \cite{DeeGeoRelKhom}; the setup is as follows. Let $X$ be a finite CW-complex, $\phi:B_1 \rightarrow B_2$ be a unital $*$-homomorphism between unital $C^*$-algebras and $C_{\phi}$ be the mapping cone of $\phi$. In \cite{DeeGeoRelKhom}, a geometric model of the Kasparov group $KK^*(C(X),C_{\phi})$ was constructed. We denote the resulting abelian group by $K_*(X;\phi)$. (A more detailed review of notation can be found at the end of the introduction). The main results of the present paper are as follows:
\begin{enumerate}
\item the construction of an explicit (i.e., defined at the level of cycles) isomorphism $\lambda: K_*(X;\phi) \rightarrow KK^{*+1}(C(X),C_{\phi})$ modelled on the classical topological index map;
\item the construction of an explicit isomorphism (i.e., analytic index map) from $K_*(X;\phi)$ to $KK^{*+1}(C(X),C_{\phi})$ using higher Atiyah-Patodi-Singer index theory (see \cite{PS}) in the special case when $X$ is a point and $\phi_*: K_*(B_1) \rightarrow K_*(B_2)$ is injective;
\item a proof of the equality of these two maps when they are both defined, see Theorem \ref{indBouEqu} (this is an index theorem).
\end{enumerate}
\par
The starting point for the construction in \cite{DeeGeoRelKhom} was not only the work of Baum and Douglas \cite{BD,BD2}, but also Higson and Roe \cite{HigRoeEta}.  The particular case when $\phi$ is the unital inclusion of the complex numbers into a ${\rm II}_1$-factor is relevant for $\rz$-valued index theory; in this example, the map at the level of $K$-theory is the inclusion of the integers into the reals. Further motivation for the construction of this particular geometric model can be found in the introduction of \cite{DeeGeoRelKhom}. It should also be mentioned that the isomorphism from $K_*(X;\phi)$ to $KK^{*+1}(C(X),C_{\phi})$ considered in \cite{DeeGeoRelKhom} was rather indirect. Hence the desire for an explicit isomorphism. 
\par
The construction of this isomorphism (Item (1) above) is via neat embeddings of manifolds with boundary into half-spaces.  As such, in the case when $X$ is a point, it can be viewed as analogous to the classical topological index map. Based on this analogy, we refer to the isomorphism $K_*(X;\phi) \rightarrow KK^{*+1}(C(X), C_{\phi})$ obtained via this embedding process as the topological index when $X$ is a point.
\par
There is also an analytic index (Item (2) above) defined under the condition that $\phi_*: K_*(B_1) \rightarrow K_*(B_2)$ is injective. This (rather restrictive) condition ensures that the higher Atiyah-Patodi-Singer index can be defined. It is satisfied in the special case when $\phi$ is the unital inclusion of the complex number into a ${\rm II}_1$-factor and for other examples that model geometric $K$-homology with coefficients (see \cite[Example 5.3]{DeeGeoRelKhom}).  The equality of the two index maps is the content of Theorem \ref{indBouEqu}. This index theorem is analogous to the Freed-Melrose index theorem \cite{Fre,FM}.  
\par
My motivation for considering index theory with values in the K-theory of mappings cones is based on its relationship with $\rz$-valued index theory. There are other interesting examples. Recently, see \cite{CWY}, Chang, Weinberger, and Yu have considered the following framework. Let $\pi_1$ and $\pi_2$ be two discrete finitely generated groups and $\alpha: \pi_1 \rightarrow \pi_2$ be a group homomorphism. If $C^*(\pi_i)$ denotes the full group $C^*$-algebra of $\pi_i$, then (from $\alpha$) we obtain a $*$-homomorphism $\tilde{\alpha}:C^*(\pi_1) \rightarrow C^*(\pi_2)$. Index maps (in particular, a version of the Baum-Connes assembly map) taking value in $K_*(C_{\tilde{\alpha}})$ are considered in \cite{CWY}; these constructions are analytic in nature. In joint work with Magnus Goffeng (see \cite{DeeGofGroupCone}), we explore the connection between the work in \cite{CWY} and the more geometric results of the current paper. In particular, we consider an analytically defined index map without assuming $\phi_*: K_*(B_1) \rightarrow K_*(B_2)$ is injective.
\par
The prerequisites for the paper are as follows. Beyond knowledge of the Baum-Douglas model, we assume the reader is familiar the basic properties of Hilbert $C^*$-module bundles, see for example \cite[Section 2]{Sch}; the bundles we consider are always locally trivial. Section \ref{higIndSec} builds on properties of higher Atiyah-Patodi-Singer index theory (see \cite{PS} and references therein for details).  A number of constructions considered here require the framework of $KK$-theory (see \cite{Kas}). In particular, we generalize a number of constructions from \cite{BOOSW} to our setting.
\par
Throughout the paper, $N$ denotes a ${\rm II}_1$-factor, $B_1$ and $B_2$ unital $C^*$-algebras, $\phi: B_1 \rightarrow B_2$ a unital $*$-homomorphism, $C_{\phi}$ the mapping cone of $\phi$ and $X$ a finite CW-complex. The suspension of a $C^*$-algebra, $A$, is denoted by $SA$. If $B$ is a unital $C^*$-algebra, then the $C^*$-algebra of continuous $B$-valued function on $X$ is denoted by $C(X,B)$. Finitely generated projective Hilbert $B$-module bundles over $X$ that are locally trivial will be refer to as $B$-bundles over $X$. The Grothendieck group of (isomorphism classes of) $B$-bundles over $X$ is denoted by $K^0(X;B)$.  It is well-known (for example, \cite[Proposition 2.17]{Sch}) that 
$$K^0(X;B) \cong K_0(C(X,B))\cong K_0(C(X)\otimes B).$$ 
Given a $B$-bundle over $X$, we have classes $[E] \in K^0(C(X)\otimes B)$ and $[[E]]\in KK^0(C(X), C(X)\otimes B)$ (see for example \cite[Section 3.4]{Rav}).
\par
A cycle in the Baum-Douglas model is a triple, $(M,E,f)$, where $M$ is a compact ${\rm spin^c}$-manifold, $E$ is a vector bundle over $M$, and $f:M\rightarrow X$ is a continuous map; we let $K^{geo}_*(X)$ denote the abelian group obtained from these cycles. More generally, given a unital C*-algebra, $B$, one can obtain a model for $KK^*(C(X), B)$ (denoted by $K^{geo}_*(X;B)$) by replacing the vector bundle $E$ with a $B$-bundle, see \cite{Wal} for details. The precise definition of the cycles used to define $K_*(X;\phi)$ is given in Definition \ref{phiCyc}, while the group is defined in Definition \ref{equCyc}. The topological index is denoted by ${\rm ind}_{top}$. Subscript notation is also used in the case of Dirac type operators to specify which manifold it is acting on and if it is twisted by a $B$-bundle.

\section{Review of the geometric model}
\noindent 
We review the constructions and main results of \cite{DeeGeoRelKhom}.
\begin{define}
Let $W$ be a locally compact space, $Z$ a closed subspace of $W$, and $\phi:B_1 \rightarrow B_2$ a unital $*$-homomorphism between unital $C^*$-algebras.  Then
$$C_0(W,Z;\phi):=\{ (f,g) \in C_0(W,B_2) \oplus C_0(Z,B_1) \: | \: f|_Z = \phi \circ g \}$$
\end{define}
We note that $C_0(W,Z;\phi)$ is a $C^*$-algebra; it fits into the following pullback diagram:
\begin{center}
$\begin{CD}
C_0(W,Z;\phi) @>>> C_0(Z,B_1) \\
@VVV  @V\phi_* VV \\
C_0(W,B_2) @>|_{Z}>> C_0(Z,B_2) \\ 
\end{CD}$
\end{center}
A prototypical example is the case when $W$ is a manifold with boundary and $Z=\partial W$.  In particular, the mapping cone of $\phi$ (denoted by $C_{\phi}$) is obtained by taking $W=[0,1)$ and $Z=pt$. The $K_0$-group of $C_0(W,Z;\phi)$ is denoted by $K^0(W,Z;\phi)$.  If $g:W \rightarrow W^{\prime}$ is a continuous map such that $g(Z) \subseteq Z^{\prime}$, then we obtain a $*$-homomorphism, $\tilde{g}:C_0(W^{\prime},Z^{\prime};\phi) \rightarrow C_0(W,Z;\phi)$ and hence a map at the level of $K$-theory groups.  We also have a $K^0(W)$-module structure on $K^0(W,Z;\phi)$ obtained via 
$$ g\cdot (f_W,f_Z) := (g\cdot f_W, g|_{Z} \cdot f_Z)$$
where $g\in C(W)$ and $(f_W,f_Z) \in C_0(W,Z;\phi)$. We will also make use of 
$$C_b(W,Z, \phi):=\{ (f,g) \in C_b(W,B_2) \oplus C_b(Z,B_1) \: | \: f|_Z = \phi \circ g \}$$
where $C_b(X,B)$ denotes the bounded $B$-valued function on $X$. Of course, if $W$ is compact, then $C_b(W,Z;\phi)=C_0(W,Z;\phi)$.
\begin{define} {\bf Cycles with vector bundle data} \label{cycBunDat} \cite[Definition 4.2]{DeeGeoRelKhom} \\ 
A cycle (over $X$ with respect to $\phi$ using bundle data) is given by, $(W, (E_{B_2},F_{B_1},\alpha),f)$, where 
\begin{enumerate}
\item $W$ is a smooth, compact ${\rm spin^c}$-manifold with boundary;
\item $E_{B_2}$ is a smooth $B_2$-bundle over $W$;
\item $F_{B_1}$ is a smooth $B_1$-bundle over $\partial W$;
\item $\alpha : E_{B_2}|_{\partial W} \cong \phi_*(F_{B_1}):=F_{B_1}\otimes_{\phi}B_2$ is an isomorphism of $B_2$-bundles;
\item $f:W \rightarrow X$ is a continuous map.
\end{enumerate}
\end{define}
\begin{define} {\bf Cycles with $K$-theory data} \cite[Definition 4.3]{DeeGeoRelKhom} \label{phiCyc} \\
A cycle (over $X$ with respect to $\phi$ using $K$-theory data) is a triple, $(W,\xi,f)$, where:
\begin{enumerate}
\item $W$ is a smooth, compact ${\rm spin^c}$-manifold with boundary;
\item $\xi \in K^0(W,\partial W;\phi)$;
\item $f:W \rightarrow X$ is a continuous map.
\end{enumerate}
\end{define}
The manifold, $W$, in a cycle need not be connected. We also let $\xi_{\partial W}$ and $\xi_{W}$ denote the images of $\xi$ under the maps $p_1:K^0(W,\partial W;\phi) \rightarrow K^0(\partial W;B_1)$ and $p_2:K^0(W,\partial W;\phi) \rightarrow K^0(W;B_2)$ respectively. The opposite of a cycle, $(W,\xi,f)$, is the same data except $W$ is given the opposite ${\rm spin^c}$-structure.  It is denote by $-(W,\xi,f)$.  The disjoint union of cycles, $(W,\xi,f)$ and $(\tilde{W},\tilde{\xi},\tilde{f})$ is given by the cycle $(W\dot{\cup}\tilde{W},\xi \dot{\cup} \tilde{\xi},f\dot{\cup} \tilde{f}) $. Two cycles, $(W,\xi,f)$ and $(\tilde{W},\tilde{\xi},\tilde{f})$ are isomorphic if there exists a diffeomorphism, $h:W \rightarrow \tilde{W}$ such that $h$ preserves the ${\rm spin^c}$-structure, $h^*(\tilde{\xi})=\xi$, and $\tilde{f} \circ h=f$.  Throughout, a ``cycle" more precisely refers to an isomorphism class of a cycle.
\begin{define} \label{phiBor}
A bordism (with respect to $X$ and $\phi$) is given by $(Z,W,\eta,g)$ where 
\begin{enumerate}
\item $Z$ is a compact ${\rm spin^c}$-manifold with boundary;
\item $W \subseteq \partial Z$ is a regular domain (see for example \cite[Definition 4.4]{DeeGeoRelKhom});
\item $\eta \in K^0(Z,\partial Z - {\rm int}(W);\phi)$;
\item $g:Z \rightarrow X$ is a continuous map.  
\end{enumerate}
\label{borBou}
The ``boundary" of a bordism, $(Z,W,\eta,F)$, is given by $(W,\eta|_W,g|_W)$. This notion of bordism leads to an equivalence relation which we denote by $\sim_{bor}$.
\end{define}
\begin{define}
Let $(W,\xi,f)$ be a cycle and $V$ a ${\rm spin^c}$-vector bundle of even rank over $W$.  Then, the vector bundle modification of $(W,\xi,f)$ by $V$ is defined to be $(W^V,\pi^*(\xi)\otimes_{\field{C}}\beta_V,f\circ \pi) $ where
\begin{enumerate}
\item ${\bf 1}$ is the trivial real line bundle over $W$ (i.e., $W\times \field{R}$);
\item $W^V=S(V\oplus {\bf 1})$ (i.e., the sphere bundle of $V\oplus {\bf 1}$);
\item $\beta_V$ is the ``Bott element" in $K^0(W^V)$ (see \cite[Section 2.5]{Rav});
\item $\otimes_{\field{C}}$ denotes the $K^0(W^V)$-module structure of $K^0(W^V,\partial W^V;\phi)$;
\item $\pi:W^V \rightarrow W$ is the bundle projection.
\end{enumerate}
The vector bundle modification of $(W,\xi,f)$ by $V$ is often denoted by $(W,\xi,f)^V$.
\end{define}  
\begin{define} \label{equCyc}
Let $\sim$ be the equivalence relation generated by bordism and vector bundle modification and let 
$$K_*(X;\phi)=\{ (W,\xi,f) \}/\sim$$
A cycle $(W,\xi,f)$ is even (resp. odd) if the connected components of $W$ are all even (resp. odd) dimensional.  Then, $K_0(X;\phi)$ is even cycles modulo $\sim$ and $K_1(X;\phi)$ is likewise only with odd cycles.
\end{define}
\begin{theorem} (see \cite[Proposition 4.13 and Theorem 4.19]{DeeGeoRelKhom}) \label{bockTypeSeq} \\
The set $K_*(X;\phi)$ with the operation of disjoint union is an abelian group. Moreover, if $X$ is a finite CW-complex, then the following sequence is exact: 
\begin{center}
$\begin{CD}
K_0(X;B_1) @>\phi_*>> K_0(X;B_2) @>r>> K_0(X;\phi) \\
@AA\delta A @. @VV\delta V \\
K_1(X;\phi) @<r<<  K_1(X;B_2) @<\phi_*<< K_1(X;B_1) 
\end{CD}$
\end{center}
where the maps are defined as follows:
\begin{enumerate}
\item $\phi_* : K_*(X;B_1) \rightarrow K_*(X;B_2)$ takes a cycle $(M,F_{B_1},f)$ to $(M,\phi_*(F_{B_1}),f)$. 
\item $r: K_*(X;B_2) \rightarrow K_*(X;\phi)$ takes a cycle $(M,E_{B_2},f)$ to $(M,(E_{B_2},\emptyset, \emptyset),f)$.  
\item $\delta : K_*(X;\phi) \rightarrow K_{*+1}(X;B_1)$ takes a cycle $(W,(E_{B_2},F_{B_1},\alpha),f)$ to $(\partial W,F_{B_1},f|_{\partial W})$.  
\end{enumerate}
\end{theorem}

\section{An index map via the mapping cone and imbeddings}
Let $H^{n}=\{(x_1,\ldots,x_{n})\in \field{R}^{n}\: | \: x_n\ge 0\}$.  The next lemma is a consequence of Bott periodicity and the definitions of the objects involved; its proof is left to the reader. 
\begin{lemma}
Let $k\in \field{N}$ and $X$ a finite CW-complex. Then 
\begin{eqnarray*}
KK^0(C(X),C_0(H^{2k},\field{R}^{2k-1};\phi)) & \cong & KK^0(C(X),SC_{\phi}) \\
KK^0(C(X),C_0(H^{2k+1},\field{R}^{2k};\phi)) & \cong & KK^0(C(X),C_{\phi}) 
\end{eqnarray*}
In particular, $K^0(H^{2k},\field{R}^{2k-1};\phi)\cong K^0(SC_{\phi})$ and $K^0(H^{2k+1},\field{R}^{2k};\phi)\cong K^0(C_{\phi})$.
\end{lemma}
\begin{define}
Let $W$ and $W^{\prime}$ be ${\rm spin^c}$-manifolds with boundary with dimensions equal modulo two and $i:W \rightarrow W^{\prime}$ a $K$-oriented neat embedding.  The push-forward map induced by $i$ (denoted $i!$) is given by the composition of the Thom isomorphism and the map given by identifying the normal bundle associated with $i$ with a neighbourhood of $W^{\prime}$.  Thus, the push-forward of $i$ defines a map
$$i!: K^0(W,\partial W;\phi) \rightarrow K^0(W^{\prime},\partial W^{\prime};\phi)$$
\end{define}
This map has two important properties. Firstly, as a map from cocycles of the form, $(E_{B_2},F_{B_1},\alpha)$ (see Definition \ref{cycBunDat} and \cite{DeeGeoRelKhom}) to $K$-theory classes, it is given as follows:
\begin{eqnarray}
i!(E_{B_2},F_{B_1},\alpha) & \mapsto & [((\pi_W)^*(E_{B_2})\otimes \beta_W, (\pi_{\partial W})^*(F_{B_1})\otimes \beta_{\partial W}, \tilde{\alpha}\otimes id)] \\
& & - [((\pi_W)^*(E_{B_2})\otimes \tilde{\beta}_W, (\pi_{\partial W})^*(F_{B_1})\otimes \tilde{\beta}_{\partial W} , \tilde{\alpha}\otimes id)]  \label{pushForActCocyc}
\end{eqnarray}
where 
\begin{enumerate}
\item $\pi_W$ (resp. $\pi_{\partial W}$) is the projection map from the normal bundle (resp. normal bundle restricted to the boundary) to $W$ (resp. $\partial W$);
\item $[\beta_W]-[\tilde{\beta}_W]$ is the Thom class of a normal bundle of $W$ inside $W^{\prime}$ and $\beta_{\partial W}$ (resp. $\tilde{\beta}_{\partial W}$) is the restriction of $\beta_W$ (resp. $\tilde{\beta}_W$) to the boundary. The reader should note that the bundles which form the Thom class are not unique, but the resulting $K$-theory class (i.e., the image of the map $i!$) is unique;
\item $\tilde{\alpha}$ is the isomorphism from $(\pi_W)^*(E_{B_2})|_{\partial W^{\prime}}$ to $(\pi_{\partial W})^*(F_{B_1})\otimes_{\phi} B_2$ given by 
$$(w,e) \mapsto (w, \alpha(e))$$
Notice that the range of this map is, in fact, $(\pi_{\partial W})^*(F_{B_1}\otimes_{\phi} B_2)$.  However, this bundle can be identified with $(\pi_{\partial W})^*(F_{B_1})\otimes_{\phi} B_2$; 
\end{enumerate}
Secondly, the map can be realized via the Kasparov product with an element in $KK^0(C_0(W,\partial W;\phi),C_0(W^{\prime},\partial W^{\prime};\phi))$.  The construction of this element is as follows.  Let $\nu_{W}$ be a normal bundle for $i(W)\subseteq W^{\prime}$.  Then,
\begin{equation}
i!:= ((\beta \otimes_{\field{C}} [\tilde{\pi}]) \otimes_{C_0(\nu_W) \otimes C_0(\nu_W,\partial \nu_W ;\phi)} [\iota]) \otimes_{C_0(\nu_W,\partial \nu_W;\phi)} [\theta]  \label{pushForwardRel}
\end{equation}
where 
\begin{enumerate}
\item $\beta \in KK(\field{C},C_0(\nu_{W}))$ is the Thom class.  It is defined in \cite[Appendix 4]{BOOSW}; note that we are using the $K$-theory class rather than the class in $KK(C(W),C_0(\nu_W))$.
\item $[\tilde{\pi}] \in KK(C_0(W,\partial W; \phi),C_b(\nu_W,\partial \nu_W;\phi))$ is the $KK$-theory class obtained from the $*$-homomorphism $\tilde{\pi}:C_0(W,\partial W; \phi) \rightarrow C_b(\nu_W,\partial \nu_W;\phi)$ defined via $(f_W,g_W) \mapsto (f_W\circ \pi, g_W\circ \pi|_{\partial \nu_W})$ where $\pi:\nu_W \rightarrow W$ is the bundle projection.
\item $[\iota] \in KK(C_0(\nu_W)\otimes C_b(\nu_W,\partial \nu_W;\phi),C_0(\nu_W,\partial \nu_W;\phi))$ is the $KK$-theory class obtained from the $*$-homomorphism 
$$\iota: C_0(\nu_W) \otimes C_b(\nu_W,\partial \nu_W;\phi) \rightarrow C_0(\nu_W,\partial \nu_W;\phi)$$ 
defined via $h \otimes(f_{\nu_W},g_{\partial \nu_W}) \mapsto (h\cdot f_{\nu_W},h|_{\partial \nu_W}\cdot g_{\partial \nu_W})$; here $\cdot$ denotes pointwise multiplication.
\item $[\theta] \in KK(C_0(\nu_W,\partial \nu_W;\phi),C_0(W^{\prime},\partial W^{\prime};\phi))$ is the $KK$-theory class obtained from the $*$-homomorphism $\theta: C_0(\nu_W,\partial \nu_W;\phi) \rightarrow C_0(W^{\prime},\partial W^{\prime};\phi)$ given by extension by zero.
\end{enumerate}
The reader familiar with pullbacks for $C^*$-algebras will notice that the definitions of the $*$-homomorphisms above (e.g., $\tilde{\pi}$, $\iota$, and $\theta$) are obtained naturally from the fact that the $C^*$-algebras involved are pullbacks.  We will often suppress the algebras over which the Kasparov products are taken and use subscript notation when more than one push-forward map is required. In this notation, Equation \ref{pushForwardRel} takes the form
$$i!= (\beta_{\nu_W})\otimes [\tilde{\pi}_{\nu_W}] \otimes [\iota_{\nu_W}] \otimes [\theta_{\nu_W}] $$
\begin{prop} 
\label{KKTrig}
Let $i:W \hookrightarrow W^{\prime}$ be a neat embedding.  Then, the map $i!$ is given by taking the Kasparov product with the class, $[i!]$.  Moreover, $i!$ fits into the following commutative diagram:
\begin{center}
$\minCDarrowwidth12pt\begin{CD}
@>>> K^0(W,\partial W;\phi) @>>> K^0(W;B_2)\oplus K^0(\partial W;B_1) @>r_W>> K^0(\partial W;B_2) @>>> \\
@. @Vi! VV @V i_W!\oplus i_{\partial W}! VV @V i_{\partial W}! VV @. \\
@>>> K^0(W,\partial W^{\prime};\phi) @>>> K^0(W;B_2)\oplus K^0(W^{\prime};B_1) @>r_{W^{\prime}}>>  K^0(\partial W^{\prime};B_2) @>>> 
\end{CD}$
\end{center}
The horizontal morphisms are given by $KK$-classes associated to the following $*$-homomorphisms: 
\begin{enumerate}
\item $C_0(W,\partial W;\phi) \rightarrow C_0(W,B_2)$ defined via $(f,g) \mapsto f$;
\item $C_0(W,\partial W;\phi) \rightarrow C_0(W,B_2)$ defined via $(f,g) \mapsto g$;
\item $C_0(W,B_2) \rightarrow C_0(\partial W,B_2)$ defined via $f \mapsto f|_{\partial W}$;
\item $C_0(\partial W,B_1) \rightarrow C_0(\partial W,B_2)$ defined via $f \mapsto \phi \circ f$.
\end{enumerate}
The vertical morphisms are given by the standard push-forward classes in $KK$-theory.
\end{prop}
\begin{proof}
For the proof of the first statement in the theorem, let $(E_{B_2},F_{B_1},\alpha)$ be a cocycle and let $\Gamma(M;E_A)$ denote the continous section of $E_A$ where $E_A$ is a $A$-bundle over $M$.  In this notation, the Kasparov cycle associated to $(E_{B_2},F_{B_1},\alpha)$ is given by $\xi=(\mathcal{E},\rho,0)$ where
$$ \mathcal{E}= \{ (s_W,s_{\partial W})\in \Gamma(W;E_{B_2})\oplus \Gamma(\partial W;F_{B_1}) \: | \: (s_W)|_{\partial W} = \alpha \circ (s_{\partial W} \otimes Id_{B_2})\}$$
and $\rho$ is the unital inclusion of the complex number.  The product $\xi \otimes_{C_0(W,\partial W;\phi)} [i!]$ can be explicitly computed and (as the reader can verify) is equal to the Kasparov cycle associated to the $i!(E_{B_2},F_{B_1},\alpha)$.
\par
The second statement follows from the action of $i!$ on cocycles of the form, $(E_{B_2},F_{B_1},\alpha)$, discussed above (see Equation \ref{pushForActCocyc}).   
\end{proof}
Our goal is the definition of a map, $\lambda: K_*(X;\phi) \rightarrow KK^*(C(X),SC_{\phi})$. For the even case, given a cycle $(W,\xi,f)$ in $K_0(X;\phi)$, there exists (for $k$ sufficiently large) a $K$-oriented neat embedding, $i:W \rightarrow H^{2k}$ and associated $KK$-theory element $[i!]\in KK(C_0(W,\partial W;\phi), C_0(H^{2k},\field{R}^{2k-1};\phi))$.  There are also $KK$-elements associated to $\xi$ and $f:W\rightarrow X$; namely 
\begin{enumerate}
\item $[[\xi]]:=\xi \otimes [\iota_W] \in KK(C(W),C_0(W,\partial W;\phi))$ where $[\iota_W]$ is the $KK$-theory class obtained from the $*$-homomorphism 
$$\iota: C(W) \otimes C_0(W,\partial W;\phi) \rightarrow C_0(W,\partial W;\phi)$$ 
defined via $h \otimes(f_{W},g_{\partial W}) \mapsto (h\cdot f_{W},h|_{\partial W}\cdot g_{\partial W})$; we often denote $[\iota_W]$ by $[\iota]$;
\item $[f]\in KK(C(X),C(W))$ is the $KK$-element naturally associated to the $*$-homomorphism $\tilde{f}: C(X) \rightarrow C(W)$ induced from $f$ (i.e., $\tilde{f}(g):=g\circ f$);
\end{enumerate}
Combining these three $KK$-theory elements gives the desired map.  More precisely, we have the following definition. 
\begin{define}
Let $\lambda: K_0(X;\phi) \rightarrow KK^0(C(X),SC_{\phi})$ be the map defined at the level of cycles via  
$$\lambda(W,\xi,f) := [f] \otimes_{C(W)} [[\xi]] \otimes_{C_0(W,\partial W;\phi)} [i!] \otimes_{C_0(H^{2k},\field{R}^{2k-1};\phi)} \mathcal{B} $$
where $\mathcal{B}$ denotes the $KK$-theory class which gives the map 
$$KK(C(X),C_0(H^{2k},\field{R}^{2k-1};\phi))\cong KK(C(X),C_0(H^{2},\field{R};\phi))=KK(C(X),SC_{\phi})$$ 
obtained via Bott periodicity.  The map from $K_1(X;\phi)$ to $KK(C(X),C_{\phi})$ is defined in a similar way; one uses a neat embedding into $H^{2k+1}$ (for $k$ sufficiently large).  Since Bott periodicity is a natural isomorphism, we often omit the map induced from $\mathcal{B}$.
\end{define}
A proof that the map $\lambda$ is well-defined is required.  It is standard to show that the map is well-defined at the level of cycles (i.e., independent of the choice of embedding, normal bundle, etc).  That it respects the equivalence relation used to define $K_*(X;\phi)$ is more involved. 
\par
In particular, further notation and three lemmas are required.  The first two lemmas are based on \cite[Lemmas 3.5 and 3.6]{BOOSW} (the proof of the latter is in Appendix B.2 of \cite{BOOSW}).  As such, the proofs of the lemmas stated here are similar to those for these lemmas.  The final lemma concerns functorial properties of the push-forward.  Again, the proofs is similar to the standard case. The fact that the maps are embeddings simplifies the proofs of these lemmas.  
\begin{lemma}  Let $(W,\xi,f)$ be a cycle in $K_*(X;\phi)$ and $V$ an even rank spin$^c$-vector bundle over $W$.  Also let $s:W \rightarrow  S(V\oplus {\bf 1})$ be the north-pole section of $W$ into $S(V\oplus {\bf 1})$ (i.e., $s(w):=(z(m),1)\in S(V\oplus {\bf 1})$ where $z$ is the zero section). Then, 
$$(W,\xi, f)^V = (S(V\oplus {\bf 1}),s!(\xi),f\circ \pi)$$
\end{lemma}
\begin{proof}
Denote $S(V\oplus {\bf 1})$ by $Z$.  The vector bundle, $V$, gives a normal bundle of $s(W) \subseteq Z$.  Therefore, 
$$s!=([F_V]-[F^{\infty}_V]) \otimes [\tilde{\pi}_V] \otimes [\iota_V] \otimes [\theta_V]$$
where $F_V$ and $F^{\infty}_V$ are the vector bundle used to define the Thom isomorphism (see \cite[Proposition A.10]{BOOSW} for details).  \par
The $K$-theory class associated to the cycle $(W,\xi,f)^V$ is given by 
\begin{eqnarray*}
\pi^*_Z(\xi)\cdot ([F_Z]-[F^{\infty}_Z]) & = & \xi \otimes [\tilde{\pi}_Z] \otimes ([F_Z]-[F^{\infty}_Z]) \otimes [\iota_Z] \\
& = & \xi \otimes ([F_Z]-[F^{\infty}_Z]) \otimes [\tilde{\pi}_Z] \otimes [\iota_Z] \\
& = & \xi \otimes ([F_V]-[F^{\infty}_V])\otimes [\varphi] \otimes [\tilde{\pi}_Z] \otimes [\iota_Z]
\end{eqnarray*}
where $\varphi: C_0(V) \rightarrow C(Z)$ is the natural inclusion.  The reader can check that
$$ \iota_Z \circ (id \otimes \tilde{\pi}_Z) \circ (\varphi \otimes id)= \theta_V \circ \iota_V \circ (id \otimes \tilde{\pi}_V)$$
as $*$-homomorphisms from $C_0(V)\otimes C_0(W,\partial W;\phi)$ to $C_0(Z,\partial Z;\phi)$.  The equality of these $*$-homomorphisms implies that 
$$[\varphi] \otimes [\tilde{\pi}_Z] \otimes [\iota_Z]=[\tilde{\pi}_V] \otimes [\iota_V] \otimes [\theta_V]$$
This implies the result.
\end{proof}
\begin{lemma} \label{lemUsingBounded}
Let $W$ and $W^{\prime}$ be smooth, compact $spin^c$-manifolds with boundary, $i:(W,\partial W) \rightarrow (W^{\prime},\partial W^{\prime})$ be a neat embedding and $\xi \in K^0(W,\partial W;\phi)$.  Then,
$$[[(\xi \otimes_{C_0(W,\partial W;\phi)} [i!])]]=[i] \otimes_{C(W)} [[\xi]] \otimes_{C_0(W,\partial W;\phi)} [i!]$$
\end{lemma}
\begin{proof}
The reader shoud recall that by definition
$$[[\xi]]=[\iota_W(\xi)]=[\xi]\otimes_{C_0(W,\partial W;\phi)} [\iota_W]$$
As such, we must show that
$$(\xi \otimes_{C_0(W,\partial W;\phi)} [i!])\otimes_{C_0(W^{\prime},\partial W^{\prime};\phi)} [\iota_{W^{\prime}}]=[i] \otimes_{C(W)} ([\xi]\otimes_{C_0(W,\partial W;\phi)} [\iota_W]) \otimes_{C_0(W,\partial W;\phi)} [i!]$$
\par
Let $p:\field{C} \rightarrow C(W)$ denote the $*$-homomorphism defined via $\lambda\in \field{C} \mapsto \lambda\cdot 1_W$.  It follows from the commutivity of the Kasparov product over $\field{C}$ and direct calculation that
$$\xi  = [p]\otimes \iota_W(\xi) $$
where $[p]\in KK^0(\field{C},C(W))$ is the $KK$-class associated to $p$. 
\par
Thus, $\iota_{W^{\prime}}(\xi \otimes i!)=\iota_{W^{\prime}}([p]\otimes \iota_W(\xi) \otimes i!)$. It follows that if $\iota_W(\xi) \otimes i!=(E,\rho,T)$, then $\iota_{W^{\prime}}(\xi \otimes i!)=(E,\rho^{\prime},T)$ where $\rho^{\prime}$ is the composition of the inclusion $C(W^{\prime}) \rightarrow C_0(W^{\prime},\partial W^{\prime};\phi)$ and right action of $C_0(W^{\prime},\partial W^{\prime};\phi)$. 
\par
The details are as follows.  The Hilbert module in the $KK$-cycle $\iota_{W^{\prime}}([p]\otimes \iota_W(\xi) \otimes i!)$ is given by 
$$(C(W)\otimes C(W^{\prime})) \otimes_{C(W)\otimes C(W^{\prime})} (E \otimes_{\field{C}} C(W^{\prime})) \otimes_{\iota_{W^{\prime}}} C_0(W^{\prime},\partial W^{\prime};\phi)$$
As the reader can verify, the map defined on elementary tensors via 
$$f_W \otimes g_{W^{\prime}} \otimes e \otimes h_{W^{\prime}} \otimes a \mapsto f_W \cdot e \cdot (g_{W^{\prime}}h_{W^{\prime}}a)$$
gives a Hilbert $C_0(W^{\prime},\partial W^{\prime};\phi)$-module isomorphism to $E$.  Moreover, the representation of $C(W^{\prime})$ on $E$ is the composition of the inclusion $C(W^{\prime}) \rightarrow C_0(W^{\prime},\partial W^{\prime};\phi)$ and right action of $C_0(W^{\prime},\partial W^{\prime};\phi)$.  The operator $T$ in the original Kasparov cycle for $\iota_W(\xi) \otimes i!$ also respects this Hilbert module isomorphism.
\par
To proceed further, additional notation is required. Given a locally compact space $Y$ and $C^*$-algebra $A$, let $C_b(Y;A)$ be the continuous bounded $A$-valued functions on $Y$ and 
$$C_b(\nu_W,\partial \nu_W;\phi):= \{ (f,g)\in C_b(\nu_W;B)\oplus C_b(\partial \nu_W;A) | f|_{\partial \nu_W} =\phi \circ g \}$$
Let $\pi_{\nu_W}:\nu_W \rightarrow W$ denote the projection map and $\rho_0:C(W) \rightarrow C_b(\nu_W)$ denote the $*$-homomorphism given by $f \mapsto f\circ \pi_{\nu_W}$. 
\par
Using the definition of $i!$, the class in $KK$-theory, $\xi \otimes [\iota_W] \otimes i!$, can be represented by a Kasparov cycle, $(E,\rho,T)$, with the following properties:
\begin{enumerate}
\item $E$ is a Hilbert $C_0(\nu_W,\partial \nu_W;\phi)$-module (since the Hilbert module in the definition of $i!$ is constructed from a Hilbert $C_0(\nu_W,\partial \nu_W;\phi)$-module and the inclusion $\theta:C_0(\nu_W,\partial \nu_W;\phi) \rightarrow C_0(W^{\prime},\partial W^{\prime};\phi)$);
\item $T$ commutes with the action of $\sigma:C_b(\nu_W)\rightarrow \mathcal{L}(E)$ via multipliers of $C_0(\nu_W)$;
\item The map $C(W) \rightarrow \mathcal{L}(E)$ is induced from $\psi_0:C(W) \rightarrow C_b(\nu_W)$.
\end{enumerate} 
Let $h:\nu_W \times [0,1] \rightarrow \nu_W$ be the map defined by $(x,t) \rightarrow tx$.  Then 
$$\rho_t:C(W^{\prime}) \stackrel{{\rm restriction}}{\rightarrow} C_b(\nu_W) \stackrel{\circ h(\cdot,t)}{\rightarrow} C_b(\nu_W) \stackrel{\sigma}\rightarrow \mathcal{L}(E)$$ 
defines a homotopy from $\psi_0 \circ i\circ \sigma$ to the restriction map $C(N)\rightarrow C_b(\nu_W)$ composed with $\sigma$.  These three properties imply that $(E,\rho_t,T)$ is a $KK$-homotopy from $(E,\rho \circ i, T)$ and $(E,\rho^{\prime},T)$.
\end{proof}
\begin{lemma}Let $(W,\partial W)$, $(W^{\prime},\partial W^{\prime})$ and $(\tilde{W},\partial \tilde{W})$ be smooth $spin^c$-manifolds.  If $s:(W,\partial W) \rightarrow (W^{\prime},\partial W^{\prime})$ and $i:(W^{\prime},\partial W^{\prime}) \rightarrow (\tilde{W},\partial \tilde{W})$ are neat embeddings, then 
$$[s!]\otimes_{C_0(W^{\prime},\partial W^{\prime})} [i!] = [(i\circ s)!] \in KK(C_0(W,\partial W;\phi),C_0(\tilde{W},\partial \tilde{W};\phi))$$
\end{lemma}
\begin{proof}
We leave the proof to the reader.  In fact, we will only need a weaker result: If $\xi\in K^0(W,\partial W;\phi)$, then 
$$(\xi \otimes \iota_{W})\otimes (i\circ s)!= (\xi \otimes \iota_{W}) \otimes (s! \otimes i!) $$
This equality follows from a short $KK$-theory computation using the fact that the push-forward is functorial on $K$-theory and the previous lemma. 
\end{proof}
\begin{prop}
Let $(W,\xi,f)$ be a cycle in $K_*(X;\phi)$ and $V$ a $spin^c$-vector bundle over $W$ with even dimensional fibers.  Then $$\lambda((W,\xi,f)^V)=\lambda(W,\xi,f) \hbox{ in }KK^*(C(X),SC_{\phi})$$
\end{prop}
\begin{proof}
Let $Z=S(V\oplus {\bf 1})$, $i_Z:Z\rightarrow H^n$ be a neat embedding (we take $n$ even for even cycles and $n$ odd for odd cycles), $s:W \rightarrow Z$ be the neat embedding of $W$ into $Z$ via the north pole section of $Z$, and $\pi: V \rightarrow W$ denote the projection map.  The definition of $\lambda$, the fact that $ \pi \circ s =id_W$, and the previous three lemmas imply that 
\begin{eqnarray*}
\lambda((W,\xi,f)^V) & = & [f]\otimes [\pi] \otimes [[s!(\xi)]]\otimes [i_Z!] \\
& = & [f]\otimes [\pi] \otimes \iota_Z(s!(\xi))\otimes [i_Z!] \\
& = & [f] \otimes [\pi] \otimes [s] \otimes \iota_W(\xi) \otimes [s!] \otimes [i_Z!] \\
& = & [f] \otimes [\pi\circ s] \otimes \iota_W(\xi) \otimes [(i_Z \circ s)!] \\
& = & [f] \otimes [[\xi]] \otimes [(i_Z \circ s)!] \\
& = & \lambda(W,\xi,f) 
\end{eqnarray*}
The last equality follows since $i_Z \circ s$ is a neat embedding (of $W$ into $H^n$) and the independence of the definition of $\lambda$ on the choice of embedding.
\end{proof}
The bordism relation is considered next, but first some additional notation is introduced.  Recall that 
$$H^{2k}=\{ (x_1,\ldots,x_{2k}) \in \field{R}^{2k} \: |\: x_{2k}\ge 0\} $$
and let 
$$H^{2k}_-:=\{ (x_1,\ldots,x_{2k}) \in \field{R}^{2k}\: |\: x_{2k}\le 0\} $$
We will make use of the $C^*$-algebras $C_0(H^{2k},\field{R}^{2k-1};\phi)$ and $C_0(\field{R}^{2k},H^{2k}_-;\phi)$ along with the natural maps
\begin{enumerate}
\item $R:C_0(\field{R}^{2k},H^{2k}_-;\phi) \rightarrow C_0(H^{2k},\field{R}^{2k-1};\phi)$ defined by restriction;
\item $I:C_0(\field{R}^{2k};B_2) \rightarrow C_0(\field{R}^{2k},H^{2k}_-;\phi)$ defined via $f \mapsto (\tilde{f},0)$ where 
$$\tilde{f}= \left\{ \begin{array}{ccc} f(x) & : & x\in H^{2k} \\ 0  & : & x \in H^{2k}_- \end{array}\right.$$
(the well-definedness of $\tilde{f}$ follows from the fact that $f$ vanishes at $\infty$);
\item $\tilde{I}:C_0(\field{R}^{2k};B_2) \rightarrow C_0(H^{2k},\field{R}^{2k-1};\phi)$ defined via $(f,g) \mapsto (f,0)$.
\end{enumerate}
It follows from these definitions that $R\circ I =\tilde{I}$. 
\begin{prop} \label{borTopMap}
If $(W,\xi,f)$ is a boundary in the sense of Definition \ref{phiBor}, then $\lambda(W,\xi,f)$ is trivial in $KK^*(C(X),SC_{\phi})$.
\end{prop}
\begin{proof}
We prove the result for even cycles; the odd case is similar.  The reader should recall the notation introduced immediately before the proposition. Let $(W,\xi,f)$ be a cycle in $K_0(X;\phi)$ which is the boundary of $((Z,W),\eta_Z,g)$.  Fix an embedding $j:\partial Z \hookrightarrow \field{R}^{2k}$ such that the restriction of $j$ to $W\subseteq \partial Z$ is a neat embedding of $W\rightarrow H^{2k}$.  Denote $j|_W$ by $i$.  Let $\nu_j$ be a normal bundle for $j(\partial Z) \subseteq \field{R}^{2k}$.  Then $\nu_i:=\nu_j |_{H^{2k}}$ is a normal bundle for $i(W) \subseteq H^{2k}$. 
\par 
By definition, $\lambda(W,\xi,f)=[f]\otimes [[\xi]]\otimes [i!] \in KK^0(C(X),C_0(H^{2k},\field{R}^{2k-1};\phi))$.  Let $(M,\eta,h)$ denote $(\partial Z,(\eta_Z)_{B_2}|_{\partial  Z},g|_{\partial Z})$ and $$\lambda_{B_2}(M,\eta,h):=[h]\otimes_{C(M)} [[\eta]]\otimes_{C(M)} [j!]\in KK^0(C(X),C_0(\field{R}^{2k})\otimes B_2)$$  
Standard results (see for example, \cite{Wal}) imply that $\lambda_{B_2}$ is a well-defined map from $K_0(X;B)$ to $KK^0(C(X),B_2)$.  In particular, $\lambda_{B_2}$ vanishes on boundaries.  Hence $\lambda_{B_2}(M,\eta,h)=0$ (since $(M,\eta,h)$ is a boundary in $K_*(X;B_2)$).  This observation reduces the proof to showing that 
\begin{equation} 
\lambda(W,\xi,f)=\tilde{I}_*(\lambda_{B_2}(M,\eta,h)) \label{isBoundaryBordismProof}
\end{equation}
where $\tilde{I}_*:KK^0(C(X),C_0(\field{R}^{2k})\otimes B_2) \rightarrow KK^0(C(X),C_0(H^{2k},\field{R}^{2k-1};\phi))$ is the map on $KK$-theory induced from the $*$-homomorphism, $\tilde{I}$. 
\par
Let $N\in \field{N}$ be sufficiently large so that the normal bundle $\nu_j$ translated by $(0,\ldots,0,N)$ is contained in ${\rm int}(H^{2k})$. 
For $t\in [0,1]$, let $j_t$ denote the embedding of $M$ into $\field{R}^{2k}$ defined via $j_t(m):=j(m)+(0,\ldots,0,Nt)$   For each $t$, let 
$$\lambda_{j_t}(M,\eta,h) = [h]\otimes_{C(M)} [[\tilde{\eta}_t]]\otimes_{C_0(M,W_t;\phi)} [j_t!]\in KK^0(C(X),C_0(\field{R}^{2k},H^{2k-1};\phi))$$   
where $W_t:=j_t(M)\cap H^{2k}_-$ and $[[\tilde{\eta}_t]] \in KK(C(M),C_0(M,W_t;\phi))$ the image of $\eta$ under the map induced from the $*$-homomorphism, $C_0(M,W;\phi) \rightarrow C_0(M,W_t;\phi)$ defined via $(f,g) \mapsto (f,g|_{W_t})$.
It follows from the definitions of $I$, $R$, $j_t$, etc that 
$$R_*(\lambda_{j_0}(M,\eta,h))=\lambda(W,\xi,f) \hbox{ and } \lambda_{j_1}(M,\eta,h)=I_* (\lambda_{B_2}(M,\eta,h))$$   
Moreover, $\lambda_{j_t}(M,\eta,h)$ defines a homotopy between the $KK$-cycles $\lambda_{j_0}(M,\eta,h)$ and $\lambda_{j_1}(M,\eta,h)$.  Hence
\begin{eqnarray*}
\lambda(W,\xi,f) & = &  R_*(\lambda_{j_0}(M,\eta,h)) \\
&  \sim &  R_*(\lambda_{j_1}(M,\eta,h)) \\
& = & (R \circ I)_*(\lambda_{B_2}(M,\eta,h)) \\
& = & \tilde{I}_*(\lambda_{B_2}(M,\eta,h))
\end{eqnarray*}
As noted in Equation \ref{isBoundaryBordismProof}, this implies the result.
\end{proof}
\begin{theorem} \label{IsoToKKFinCW}
If $X$ is a finite CW-complex, then the map $\lambda:K_*(X;\phi) \rightarrow KK^*(C(X),SC_{\phi})$ is an isomorphism.
\end{theorem}
\begin{proof}
The main step is to show that the following diagram commutes:
{\footnotesize \begin{center}
$\minCDarrowwidth10pt\begin{CD}
@>>> K_0(X;B_1) @>\phi_*>> K_0(X;B_2) @>r>> K_0(X;\phi) @>\delta>> K_1(X;B_1) @>>> \\
@. @V\lambda_{B_1} VV @V\lambda_{B_2} VV @V\lambda VV @V\lambda_{B_1} VV @. \\
@>>>  KK^0(C(X),B_1) @>\phi_* >> KK^0(C(X),B_2) @> r^{ana}_*>> KK^0(C(X),SC_{\phi}) @>\delta^{ana}_*>> KK^1(C(X),B_1) @>>> 
\end{CD}$
\end{center}}
where 
\begin{enumerate}
\item The first exact sequence is from Theorem \ref{bockTypeSeq};
\item The vertical maps, $\lambda_{B_i}$ ($i=1,2$), are defined at the level of cycles via $\lambda_{B_i}(M,E_{B_i},f)=[f]\otimes_{C(M)} [[E_{B_i}]]\otimes_{C(M)} [D_M]$ (see \cite{Wal} for details);
\item The second exact sequence is the long exact sequence in $KK$-theory obtained from the short exact sequence of $C^*$-algebras 
$$0 \rightarrow SB_2 \rightarrow C_{\phi} \rightarrow B_1 \rightarrow 0$$
\end{enumerate}
\par
Again, the details of commutativity are given for even cycles, but the odd case is similar.  That $\lambda_{B_2} \circ \phi_*=\phi_* \circ \lambda_{B_1}$ is standard.  With the goal of showing that $r^{ana} \circ \lambda_{B_2} = \lambda \circ r$ in mind, let $(M,E_{B_2},f)$ be a geometric cycle in $K_0(X;B_2)$.  Then 
$$(r^{ana} \circ \lambda_{B_2}) (M,E_{B_2},f)= r^{ana} ([f]\otimes [[E_{B_2}]]\otimes [i!])$$  
where $i:M \rightarrow \field{R}^{2k}$ is an embedding.  But $r^{ana}$ is given by the inclusion of $C_0(\field{R}^{2k})\otimes B_2 \rightarrow C_0(H^{2k},\field{R}^{2k-1};\phi)$.  It is induced from the natural inclusion, $\hat{r}: \field{R}^{2k} \hookrightarrow H^{2k}$. However, the map $i\circ \hat{r}$ is a (neat) embedding of $M \rightarrow H^{2k}$. Using this embedding in the definition of $\lambda$, leads to the result. 
\par
Next, the proof that $\lambda_{B_1} \circ \delta = \delta^{ana} \circ \lambda$ is considered.  Let $(W,\xi,f)$ be a cycle in $K_0(X;\phi)$ and $i:W \hookrightarrow H^{2k}$ a neat embedding.  Then 
$$\lambda_{B_1}(\delta(W,\xi,f))=\lambda_{B_1}(\partial W,\xi_{B_1},f|_{\partial W})=[f|_{\partial W}]\otimes [[\xi_{B_1}]]\otimes [i|_{\partial W}!]$$  
Whereas
$$(\delta^{ana} \circ \lambda) (W,\xi,f)=\delta^{ana} ([f] \otimes [[\xi]]\otimes [i!]) = [f] \otimes [[\xi]]\otimes [i!] \otimes [ev_{\field{R}^{2k}}]$$
where $ev_{\field{R}^{2k}}: C_0(H^{2k},\field{R}^{2k-1};\phi) \rightarrow C_0(\field{R}^{2k-1} )\otimes B_1$ is given by $(f,g) \rightarrow g$.  To compare these $KK$-classes, three $*$-homomorphisms are required; they are 
\begin{enumerate}
\item $\gamma: C_0(W,\partial W;\phi) \rightarrow C(\partial W) \otimes B_1$ is defined via $(f,g) \mapsto g$; 
\item $\iota_W : C(W) \otimes C_0(W,\partial W;\phi) \rightarrow  C_0(W,\partial W;\phi)$ is defined above in the discussion following Equation \ref{pushForwardRel};
\item $r_W: C(W) \rightarrow C(\partial W)$ is the restriction to the boundary (i.e., $r_W(f)=f|_{\partial W}$);
\end{enumerate}
The $KK$-classes associated to these $*$-homomorphisms satisfy the following 
\begin{enumerate}
\item  $[f|_{\partial W}]=[f] \otimes_{C(W)} [r_W]$; 
\item $[r_W]\otimes [[\xi_{B_1}]] = [[\xi]] \otimes \gamma$;
\item $[i!] \otimes [ev_{\field{R}^{2k}}] =[\gamma]\otimes [i|_{\partial W}!]$.  
\end{enumerate}
The proofs of these equalities follows from standard properties of $KK$-theory.  The first equality is standard.  In regards to the second (i.e., showing that $[r_W]\otimes [[\xi_{B_1}]] = [[\xi]] \otimes \gamma$), we consider the case when $\xi$ is given by a triple $(E_{B_2},F_{B_1},\alpha)$ (rather than a formal difference of such triples); the general case easily follows.  If $E_A$ is a $A$-bundle over $M$, then let $\Gamma(M;E_A)$ denote the continuous sections of $E_A$. 
\par  
Using this notation, the Kasparov cycle $[[(E_{B_2},F_{B_1},\alpha)]]$ is given by $(\mathcal{E},\rho,0)$ where
$$ \mathcal{E}= \{ (s_W,s_{\partial W})\in \Gamma(W;E_{B_2})\oplus \Gamma(\partial W;F_{B_1}) \: | \: (s_W)|_{\partial W} = \gamma \circ (s_{\partial W} \otimes Id_{B_2})\}$$
and, for $g\in C(W)$,
$$ \rho(g)\cdot (s_W,s_{\partial W}):= (g\cdot s_W, g|_{\partial W}\cdot s_{\partial W}) $$
On the other hand, the Kasparov cycle $[[F_{B_1}]]$ is given by 
$$ (\Gamma(\partial W;F_{B_1}),\varphi,0)$$
where $\varphi$ is the representation of $C(\partial W)$ via pointwise multiplication.  The Kasparov products $[r_W]\otimes [[F_{B_1}]]$ and $[[(E_{B_2},F_{B_1},\alpha)]]\otimes [\gamma]$ can be explicitly computed.  The methods used in the proof of Lemma 3.4.4 in \cite{Rav} can be used to prove the equality of these $KK$-elements; the details are left to the reader.
\par
Finally, that $[i!] \otimes [ev_{\field{R}^{2k}}] =[\gamma]\otimes [i|_{\partial W}!]$ follows from the commutative diagram considered in Remark \ref{KKTrig}.
\par
These computations imply that
\begin{eqnarray*}
[f|_{\partial W}]\otimes [[\xi_{B_1}]]\otimes [i|_{\partial W}!] & = & [f] \otimes [r_W] \otimes [[\xi_{B_1}]]\otimes [i|_{\partial W}!] \\ 
& = &  [f] \otimes [[\xi]] \otimes [\gamma] \otimes [i|_{\partial W}!] \\
& = & [f] \otimes [[\xi]]\otimes [i!] \otimes [ev_{\field{R}^{2k}}]
\end{eqnarray*}
This completes the proof that the diagram given at the beginning of the proof commutes.  The Five Lemma and the fact that $\lambda_{B_1}$ and $\lambda_{B_2}$ are isomorphisms for $X$ a finite CW-complex (see for example \cite{Wal}) then imply that $\lambda$ is also an isomorphism.  
\end{proof}
\begin{define}
Let $(W,\xi,f)$ be a cycle in $K_p(X;\phi)$ ($p=0$ or $1$).  Then, for $k$ sufficiently large, there exists, $i:W \rightarrow H^{2k+p}$, a $K$-oriented neat embedding of $W$ into the halfspace $H^{2k+p}$.  The topological index of $(W,\xi,f)$ is defined to be
$${\rm ind}_{top}(W,\xi,f):=i!(\xi)$$
Using Bott periodicity, we can (and will) consider this as an element in $K_p(SC_{\phi})$.  
\end{define}
\begin{cor} \label{topIndCorPt}
The topological index map is well-defined (as a map from $K_p(X;\phi)$ to $K_p(H^{2},\field{R};\phi)$).  Moreover, in the case when $X$ is a point, the topological index map is an isomorphism.
\end{cor}
\begin{proof}
The first statement follows from the fact that the topological index map is given by the composition, $c^* \circ \lambda$, where $c:\field{C} \rightarrow C(X)$ is the natural inclusion and $\lambda$ is the isomorphism in Theorem \ref{IsoToKKFinCW}.  The second statement follows as a special case of Theorem \ref{IsoToKKFinCW}. 
\end{proof}

\section{An index map via boundary conditions} \label{higIndSec}
Our goal is the construction of an analytic index map from $K_*(X;\phi)$ to $K_{*+1}(C_{\phi})$. This index map will be defined under the assumption that $\phi_*: K_*(B_1) \rightarrow K_*(B_2)$ is injective; an important example is the case when $\phi$ is the unital inclusion of the complex numbers into a ${\rm II}_1$-factor. We make use of higher Atiyah-Patodi-Singer index theory.  In the next subsection, we discuss the relationship between the higher Atiyah-Patodi-Singer index and vector bundle modification.  This discussion is written in a self-contained manner as its main result is of some independent interest; it also serves as an introduction to higher Atiyah-Patodi-Singer index theory and the notation required for the second subsection. The reader is directed to \cite{MP} and \cite{Wahl2} and references therein for further details on this theory.
\subsection{Higher Atiyah-Patodi-Singer index theory and vector bundle modification}    
Following \cite{Wahl2}, we introduce some notation.  Let $W$ be a connected, compact, Riemannian ${\rm spin^c}$-manifold with boundary with a product structure in a neighborhood of the boundary.  Also let $W_{cyl}$ denote the manifold obtained from $W$ by attaching a cylindrical end to the boundary of $W$. In other words, there exists $\epsilon>0$, submanifold $Z_r \subseteq W_{cyl}$, and ${\rm spin^c}$-preserving isometry $e: Z_r \rightarrow (-\epsilon, \infty) \times \partial W$ such that 
$$W=W_{cyl} -  e^{-1}((0,\infty)\times \partial W)$$
We also let $Z:=\field{R}\times \partial M$, $U_{\epsilon}:= e^{-1}((-\epsilon,0])\subseteq W$, and $p$ denote the projection $U_{\epsilon} \rightarrow \partial W$.  In an abuse of notation, we refer to $\partial W \times (0,\infty)$ when working with $e^{-1}((0,\infty)\times \partial W)$.
\par
Let $B$ be a unital $C^*$-algebra, $E_B$ be a $B$-bundle over $W$ and $S_W$ be the spinor bundle associated with the ${\rm spin^c}$-structure on $W$.  Then, $\mathcal{E}:=S_{W}\otimes_{\field{C}}E_B$ has a natural Dirac $B$-bundle structure in the sense of \cite[Section 2]{Wahl2}.  We denote the Clifford connection on this bundle by $\nabla$ and assume that this construction respects the product structure of $\partial W \subseteq W$.  In particular, $\mathcal{E}|_{U_{\epsilon}} = p^*(\mathcal{E}|_{\partial W})$.
\par
Let $\slashed{\partial}_{\partial W}$ denote the Dirac operator associated to the bundle $S_{\partial W}\otimes E_B|_{\partial W}$.  In \cite{Wahl2} (also see \cite{MP}), a number of operators are associated to the data introduced in the previous two paragraphs.  First, however, we must perturb the operator on the boundary.  Let $A$ be a selfadjoint operator in $\mathcal{B}(L^2(\partial W; S_{\partial W} \otimes (E_B|_{\partial W})))$ such that $\slashed{\partial}_{\partial W} + A$ is invertible. The existence of $A$ follows from the vanishing of the index of $\slashed{\partial}_{\partial W}$ (see \cite{MP} for further details).  In fact, we could assume that $A$ is a smoothing operator.  Following the notation of \cite{Wahl2}, let $D_W(A)$ be the operator on $W$ associated to higher Atiyah-Patodi-Singer boundary conditions and $D_{W_{cyl}}(A)$ be the Dirac operator on $W_{cyl}$ perturbed on the cylinder by $A$.  A detailed discussion of these operators (in particular, their construction) can be found in \cite[Section 2]{Wahl2}.  
\par
Since the latter operator is of more importance in this work, we only give the details of its construction. Let $\tilde{\mathcal{E}}$ denote the extension of $\mathcal{E}$ from $W$ to $W_{cyl}$, $\slashed{\partial}$ denote the Dirac operator associated to it, and $\chi:W \rightarrow [0,1]$ be a function which satisfies
\begin{enumerate}
\item ${\rm supp}(\chi)\subseteq \partial W \times (-\frac{3\epsilon}{4},\infty)$;
\item For each $w\in \partial W \times (0,\infty)$, $f(\chi)=1$;
\end{enumerate}   
Denote the Clifford action by $c$ and the coordinate in the normal direction by $x_1$.  Then $D_{W_{cyl}}(A)$ is defined to be the closure (on $L^2(W_{cyl};\tilde{\mathcal{E}})$) of the operator 
$$\slashed{\partial} - c(dx_1)\chi A$$
It has an associated index in the $K$-theory of $B$.  The reader can find further details on this construction in \cite{Wahl2}.  
\par
Our goal is to consider vector bundle modification as it relates to higher index theory for manifolds with boundary.  As such, let $V$ be a ${\rm spin^c}$-vector bundle over $W$ with even-dimensional fibers.  Further, assume that $V$ respects the product structure of $\partial W \subseteq W$.  Using the vector bundle modification operation, we obtain from $W$ and $V$ a ${\rm spin^c}$-manifold $\hat{W}:=S(V\oplus {\bf 1})$ where ${\bf 1}$ denotes the trivial real line bundle over $W$; note that $\hat{W}$ is a fiber bundle over $W$.  Moreover, since $W$ is connected, the fiber is $S^{2k}$ for some $k\in \field{N}$.  By extending the vector bundle $V$ to $W_{cyl}$, we can also consider the vector bundle modification of $W_{cyl}$.  We denote the resulting manifold by $\hat{W}_{cyl}$.
\par
The vector bundle modification operation affects the bundle data on $W$ as follows.  Let $\beta$ denote the Bott bundle over $\hat{W}$; it is a vector bundle and its construction can be found in \cite{BD}.  Then the Hilbert $B$-bundle on $\hat{W}$ is given by $\pi^*(E_B)\otimes_{\field{C}} \beta$ where $\pi:\hat{W} \rightarrow W$ is the projection map.  By the two out of three property of ${\rm spin^c}$-vector bundles (see for example \cite{BHS}), there is a ${\rm spin^c}$-structure on $\hat{W}$.  We let $S_{\hat{W}}$ denote the spinor bundle associated with the ${\rm spin^c}$-structure and $\hat{\mathcal{E}}$ denote the $B$-Dirac bundle $S_{\hat{W}}\otimes \pi^*(E_B)\otimes_{\field{C}} \beta$.  These constructions can also be applied to $\hat{W}_{cyl}$.  In an abuse of notation, we denote the Bott bundle over $\hat{W}_{cyl}$ also by $\beta$ and the $B$-Dirac bundle over $\hat{W}_{cyl}$ also by $\hat{\mathcal{E}}$.  Based on this discussion, we can construct the associated operators discussed in the preceeding paragraphs (this time on the manifolds $\hat{W}$ and $\hat{W}_{cyl}$).  However, the construction of these operators involved the choice of operator $A$.  We would like to construct from a choice of $A$ on the base $W$ a natural choice of such an operator for $\hat{W}$.  
\par
The desired construction and the main result of this subsection are the content of the next proposition.  The proof requires the following lemma which is a well-known result in $KK$-theory (cf. \cite[Lemma 2.7]{BHS} in the case of analytic $K$-homology).
\begin{lemma} \label{KKLemma}
Let $(\mathcal{E},\rho,F)$ be a Kasparov cycle representing a class in $KK^0(A,B)$ and suppose that $T\in \mathcal{L}(\mathcal{E})$ is a self-adjoint, odd-graded involution which commutes with action of $A$ and anticommutes with $F$.  Then the class (in $KK^0(A,B)$) of $(\mathcal{E},\rho,F)$ is zero.
\end{lemma}
\begin{prop} \label{APSandVBM}
We use the notation introduced in the previous few paragraphs. For example, $W$ denotes a compact $spin^c$-manifold with boundary, $E_B$ a $B$-bundle, and $V$ a $spin^c$-vector bundle over $W$ with even dimensional fibers.  Then, given a choice of Dirac operator on $W$ and selfadjoint operator $A$ (see above), there exists a Dirac operator on $\hat{W}$ and selfadjoint operator $\hat{A}$ such that  
$${\rm ind}(D_W(A))={\rm ind}(D_{\hat{W}}(\hat{A})) \in K_*(B)$$
The reader should note that the Dirac operator on $\hat{W}$ and $\hat{A}$ are defined in the proof.  
\end{prop}
\begin{remark} \label{VBMAPSRem}
A word concerning this proposition seems in order.  Perhaps most importantly, the proposition does not imply that the higher Atiyah-Patodi-Singer index is invariant under vector bundle modification.  The specific choice of spectral section and Dirac operator on the manifold $\hat{W}$ are important to the proof.  These operators are constructed via a partition of unity argument. 
\par  
In this regard, the statement of this proposition is unsatisfying in a number ways. In particular, one would hope to find a {\it canonical} construction of a spectral section on the modified manifold given one on the base; our construct of the spectral section is quite ad hoc. Despite this, the result suffices for our purposes.
\end{remark}
\begin{proof}
The structure of the proof is as follows.  By \cite[Propostion 2.1]{Wahl2}, the proof will be complete upon showing that the operators $D_{W_{cyl}}(A)$ and $D_{\hat{W}_{cyl}}(\hat{A})$ have the same index; of course, the construction of $\hat{A}$ and the Dirac operator are also required.  Apart from these constructions, the proof consists of two steps
\begin{enumerate}
\item proving the result in the case when $V$ is a trivial vector bundle.  The reader should note that in this case, $\hat{W}=W\times S^{2k}$;
\item using a partition unity argument to treat the case of general $V$.
\end{enumerate}
As such, the steps in the proof are the same as those in the proof of Proposition 3.6 in \cite{BHS}.  The case when $W$ is even dimensional is considered in detail; the odd case is left to the reader.  
\par
The case when $V$ is a trivial bundle is considered first.  In this case, $\hat{W}=W\times S^{2k}$ and we can take the product of the Dirac operators to form the Dirac operator on $W\times S^{2k}$. The identification 
$$L^2(\partial W \times S^{2k};(S_{\partial W} \otimes E_B|_{\partial W})\boxtimes \beta) \cong L^2(\partial W; S_{\partial W} \otimes E_B|_{\partial W}) \hat{\otimes} L^2(S^{2k};\beta)$$
will be used throughout. In particular, we apply it to define $\hat{A}:= (A\otimes I)$. It follows that $\hat{A}$ is selfadjoint and $\slashed{\partial}_{\partial W\times S^{2k}}+\hat{A}$ is invertible. To see that the latter of these statements holds, one (using the fact that we are working in a graded situation (see for example \cite[Section A.2]{HR})) notes that 
$$(\slashed{\partial}_{\partial W\times S^{2k}}+\hat{A})^2 = (\slashed{\partial}_{\partial W}+A)^2\otimes I + I\otimes \slashed{\partial}_{S^2k}^2 $$
where $\slashed{\partial}_{\partial W}$ and $\slashed{\partial}_{S^2k}$ are respectively the Dirac operators on $\partial W$ and $S^{2k}$. That this operator is invertible follows since $(\slashed{\partial}_{\partial W}+A)^2$ is invertible and both operators are positive. The invertiblity of the original operator follows since it is selfadjoint; in particular, 
$$(\slashed{\partial}_{\partial W\times S^{2k}}+\hat{A})^2=(\slashed{\partial}_{\partial W\times S^{2k}}+\hat{A})^*(\slashed{\partial}_{\partial W\times S^{2k}}+\hat{A})$$
\par
Let $\hat{\odot}$ denote the (graded) algebraic tensor product and $\mathcal{S}$ denote the spinor bundle of $S^{2k}$.  Then, on 
$$C^{\infty}_c(M_{cyl};\mathcal{E})\hat{\odot} C^{\infty}(S^{2k};\mathcal{S}\otimes \beta)\subset C^{\infty}_c(M_{cyl}\times S^{2k};\hat{\mathcal{E}})$$
the twisted Dirac operator on $W_{cyl}\times S^{2k}$ has the form
$$ \slashed{\partial}_{W_{cyl}} \hat{\otimes} I + I \hat{\otimes} \slashed{\partial}_{S^{2k}}$$
In fact, operator $\slashed{\partial}_{W_{cyl}\times S^{2k}} - c(dx_1) \hat{\chi} \hat{A}$ also decomposes in this way.  That is, on $C^{\infty}_c(M_{cyl};\mathcal{E})\hat{\odot} C^{\infty}(S^{2k};\mathcal{S}\otimes \beta)$, it is equal to 
$$(\slashed{\partial}_{W_{cyl}} +c(dx_1) \chi A) \hat{\otimes} I + I \hat{\otimes} \slashed{\partial}_{S^{2k}}$$
Here, the reader should note that $\hat{\chi}$ and $\chi$ are related as follows: $\hat{\chi}: W_{cyl} \times S^{2k} \rightarrow [0,1]$ is defined via $\hat{\chi}(w,z):= \chi(w)$.  The closure of the above operator (denoted by $D_{W_{cyl}\times S^{2k}}(\hat{A})$) therefore has the form  
$$D_{W_{cyl}\times S^{2k}}(\hat{A})= D_{W_{cyl}}(A)\hat{\otimes} I+ I \hat{\otimes} D_{S^{2k}} $$
as an operator on 
$$L^2(W_{cyl}\times S^{2k};\hat{\mathcal{E}}) \cong L^2(W_{cyl};\mathcal{E})\hat{\otimes}L^2(S^{2k};\mathcal{S}\otimes \beta)$$
We now apply techniques from \cite{BHS}.  Namely, the Hilbert module on which the operator $D_{W_{cyl}\times S^{2k}}(\hat{A})$ acts (as an unbounded operator) decomposes as follows
\begin{eqnarray*}
L^2(W_{cyl}\times S^{2k};\hat{\mathcal{E}}) & \cong & L^2(W_{cyl};\mathcal{E})\hat{\otimes}L^2(S^{2k};\mathcal{S}\otimes \beta) \\ 
& \cong &  (L^2(W_{cyl};\mathcal{E})\hat{\otimes}{\rm ker}(D_{S^{2k}})) \oplus (L^2(W_{cyl};\mathcal{E})\hat{\otimes} {\rm ker}(D_{S^{2k}})^{\bot})
\end{eqnarray*}
Moreover, the operator respects this decomposition.  That is, if $P$ denotes the projection onto $L^2(W_{cyl};\mathcal{E})\hat{\otimes}{\rm ker}(D_{S^{2k}})$, then
$$D_{W_{cyl}\times S^{2k}}(\hat{A})=PD_{W_{cyl}\times S^{2k}}(\hat{A})P + P^{\bot}D_{W_{cyl}\times S^{2k}}(\hat{A})P^{\bot}$$
The operator $PD_{W_{cyl}\times S^{2k}}(\hat{A})P$ acts as $D_{W_{cyl}}(A)$ on $L^2(W_{cyc};\mathcal{E})\hat{\otimes} {\rm ker}(D_{S^{2k}})$; to see this, note that ${\rm ker}(D_{S^{2k}})$ is one dimensional and is given by the span of an even section (see \cite[Proposition 3.11]{BHS}).  
\par
This reduces the proof (of the special case when $V$ is trivial) to showing that ${\rm ind}(P^{\bot}D_{W_{cyl}\times S^{2k}}(\hat{A})P^{\bot})=0$.  To this end, consider the operator $\gamma \otimes T$ where $\gamma$ is the grading operator and $T$ is the partial isometry in the polar decomposition of $D_{S^{2k}}$.  As the reader can verify (see also \cite[Section 4]{BHS}), this operator is an odd graded involution on $L^2(M;\mathcal{E})\hat{\otimes} {\rm ker}(D_{S^{2k}})^{\bot}$.  Moreover, $\gamma\otimes T$ anti-commutes with $P^{\bot}D^{prod}_{W_{cyl}\times S^{2k}}(\hat{A})P^{\bot}$.  Lemma \ref{KKLemma} implies that ${\rm ind}(P^{\bot}D_{M\times S^{2k}}P^{\bot})$ is zero.  This completes the proof in the case when $V$ is a trivial vector bundle.
\par
The general case is now considered.  As such, let $V$ be a general ${\rm spin^c}$-vector bundle with even-dimensional fibers. We must construct the Dirac operator and the operator, $\hat{A}$.  
\par
We begin with the Dirac operator on the boundary of $\hat{W}$.  Again, the reader should compare our construction with the one in the proof of Proposition 3.6 in \cite{BHS}. We denote the principal ${\rm Spin^c}(2k)$-bundle associated to the ${\rm spin^c}$-structure of $\partial \hat{W}$ by $\mathcal{P}_{\partial \hat{W}}$. The Dirac operator on the boundary (twisted by the relevant Hilbert module bundle), $D_{\partial \hat{W}}$ acts on a Hilbert module which is naturally isomorphic to 
\begin{equation} \label{nonTrivialBoundary}
(L^2(\mathcal{P}, \pi^*(S_{\partial W}\otimes E_B|_{\partial W})) \hat{\otimes} L^2(S^{2k};S_{S^{2k}}\otimes \beta))^{{\rm Spin^c}(2k)} 
\end{equation}
where
\begin{enumerate}
\item $\pi: \mathcal{P} \rightarrow \partial W$ is the projection map;
\item  $S_{W}$ and $S_{S^{2k}}$ are the spinor bundles over $\partial W$ and $S^{2k}$ respectively;
\item $\beta$ is the Bott bundle over $S^{2k}$ (for example, see \cite{BD}).
\end{enumerate}
Let $D_{S^{2k}}$ denote the Dirac operator on $S^{2k}$ twisted by the Bott bundle. In the  proof of Proposition 3.6 in \cite{BHS}, an equivariant, first order, formally self-adjoint differential operator acting on $L^2(\mathcal{P}; \pi^*(S_{\partial W}))$ is constructed. Let $R$ denote the operator obtained by twisting this operator by $\pi^*(E_B|_{\partial W})$; results in \cite{BHS} imply that
$$D_{\partial \hat{W}} = R \hat{\otimes}I + I \hat{\otimes} D_{S^{2k}}$$ 
The construction of $R$ depends on the following data (which we list and fix)
\begin{enumerate}
\item a finite open cover, $\{ U_j \}_{j\in J}$, of $\mathcal{P}$ such that $\mathcal{P}|_{U_j}$ is trivial for each $j\in J$; 
\item specific choices of trivializations, $\mathcal{P}|_{U_j} \cong {\rm spin^c}(2k)\times U_j$; 
\item a smooth partition of unity subordinate to the cover. 
\end{enumerate}
We define $\hat{A}$ to be $\pi^*(A)\hat{\otimes} I$. It is clear that $\hat{A}$ is a selfadjoint operator, but we must show that the operator 
$$D_{\partial \hat{W}} + \hat{A} = (R + \pi^*(A))\hat{\otimes}I + I \hat{\otimes} D_{S^{2k}}$$
is invertible. 
\par
The details are as follows. The operator $D_{\partial \hat{W}}$ respects the Hilbert module decomposition
$$(L^2(\mathcal{P}, \pi^*(S_{\partial W}\otimes E_B)) \hat{\otimes} L^2(S^{2k};S_{S^{2k}}\otimes \beta))^{{\rm Spin^c}(2k)} \cong \mathcal{E} \oplus \mathcal{E}^{\bot}$$
where $\mathcal{E} \cong L^2(\mathcal{P};\pi^*(S_{\partial W}\otimes E_B))^{{\rm spin^c}(2k)} \otimes {\rm ker}(D_{S^{2k}})$; the reader can find more details on this decompostion in the proof of Proposition 3.6 of \cite{BHS}. Combining this identification and the fact that ${\rm ker}(D_{S^{2k}})$ is one dimensional and given by the span of an even section (see \cite[Proposition 3.11]{BHS}), we have that $D_{\partial \hat{W}}+\hat{A}$ acts as $D_{\partial W} + A$ on the factor $\mathcal{E}$; hence, it is invertible on this factor. 
\par
On the second factor, we have that $I \hat{\otimes} D_{S^{2k}}$ is invertible (this observation uses the fact that the spectrum of $D_{S^{2n}}$ is discrete). Using an argument similar to the one used to show invertiblity in the case of a trivial $V$, it follows that the restriction of $D_{W}$ to the second factor is invertible.
\par
This completes the constructions on the boundary; for $\hat{W}_{cyl}$, we proceed as follows. Let 
$\mathcal{P}_{cyl}$ denote the principal ${\rm Spin}^c(2k)$-bundle associated with the ${\rm spin^c}$-structure of $\hat{W}_{cyl}$. The Dirac operator acts on 
$$(L^2(\mathcal{P}_{cyl}, \pi^*(S_{W_{cyl}}\otimes \tilde{E}_B)) \hat{\otimes} L^2(S^{2k};S_{S^{2k}}\otimes \beta))^{{\rm Spin^c}(2k)}$$
where 
\begin{enumerate}
\item $\pi_{W_{cyl}}: \mathcal{P}_{cyl} \rightarrow W_{cyl}$ is the projection map;
\item $S_{W_{cyl}}$ is the spinor bundle on $W_{cyl}$;
\item $\tilde{E}_B$ is the extension of $E_B$ from $W$ to $W_{cyl}$;
\item the other data (e.g., $S_{S^{2k}}$, $\beta$, etc) is as in Equation \eqref{nonTrivialBoundary} on the previous page.
\end{enumerate}
The construction of the equivariant first order formally self-adjoint differential operator from \cite{BHS} discussed above can be applied here also (see \cite{BHS} for further details); it leads to the following:
$$D_{\hat{W}}=R_{\hat{W}} \hat{\otimes} I + I \hat{\otimes} D_{S^{2k}}$$  
where 
\begin{enumerate}
\item $D_{\hat{W}}$ is the Dirac operator on $W_{cyl}$ twisted by $\tilde{E}_B$ (recall that $\tilde{E}_B$ is the extension of $E_B$ to $W_{cyl}$);
\item $R_{\hat{W}}$ is the operator constructed in the proof of Proposition 3.6 in \cite{BHS} (twisted by the relevant bundle);
\item $D_{S^{2k}}$ is the Dirac operator on $S^{2k}$ twisted by the Bott bundle.
\end{enumerate}
As in the construction on the boundary, $R_{\hat{W}}$ depends on the choice of a finite open cover, trivializations, and a  partition of unity subordinate to the cover. In addition, we require that this data is compatable with the choices made in the construction of $R$. For example, for the open cover (which we denote by $\{V_i\}_{i\in I}$) used to construct $\tilde{R}$, we assume that 
\begin{enumerate}
\item the bundle, $\mathcal{P}_{cyl}|_{V_i}$ is trivial for each $i\in I$;
\item for each $i\in I$, $V_i \cap \partial (\hat{W} \times (-\epsilon,\infty))$ is empty or equal to $U_j \times (-\epsilon,\infty)$ for some $j\in J$;  
\end{enumerate}
The reader should note that although $\hat{W}_{cyl}$ is not compact such a cover exists. Similar assumptions are required for the other data used to define $R_{\hat{W}}$.
\par
With all this data fixed, we can form operator $D_{\hat{W}}(\hat{A})=D_{\hat{W}}-c(dx_1) \hat{\chi} \hat{A}$ where $\hat{\chi}$ and $c(dx_1)$ are defined as in the case of a trivial $V$. This operator takes the form
$$R_{\hat{W}}(\hat{A})\otimes I + I \otimes D_{S^{2k}}$$
where the operator $R_{\hat{W}}(\hat{A})$ is given by $R_{\hat{W}} - c(dx_1)\hat{\chi}\pi^*(A)$. Using this decomposition, the proof given in the case when $V$ is trivial can be generalized to the case of a non-trivial $V$; the details are left to the reader.
\end{proof}

\subsection{The analytic index map}
For this development, it is more convenient to work with cycles of the form given in Definition \ref{cycBunDat} (i.e., cycles containing bundle data).  In fact, we need only consider cycles in $K_*(pt;\phi)$ since the general index map will be defined by
$${\rm ind}_{ana}: K_*(X;\phi) \rightarrow K_*(pt;\phi) \rightarrow K_{*+1}(C_{\phi}) $$  
where the first map is defined at the level of cycles via $(W,(E_{B_2},F_{B_1},\alpha),f) \mapsto (W,(E_{B_2},F_{B_1},\alpha))$ and the definition of the second map is the main objective of this section; the second map will also be denoted simply as ${\rm ind}_{ana}$.  To be precise, the geometric data considered in this section is the following.  Let 
\begin{enumerate}
\item $W$ be a compact ${\rm spin^c}$-manifold with boundary; 
\item $E_{B_2}$ be a $B_2$-bundle over $W$;
\item $F_{B_1}$ be a $B_1$-bundle over $\partial W$;
\item $\alpha: F_{B_1}\otimes_{\phi} B_2 \rightarrow E_{B_2}$ is an isomorphism;
\end{enumerate}  
The starting point for defining this index is the vanishing of index of the boundary operator (see for example \cite{LP}). We define the analytic index map from the $K_*(X;\phi)$ to $K_{*+1}(C_{\phi})$ under the assumption that
$$\phi_* : K_*(B_1) \rightarrow K_*(B_2) $$
is injective. We also note that a relevant example is the case when $\phi$ is the unital inclusion of the complex numbers into a ${\rm II_1}$-factor.
\par  
Additional geometric data must be fixed to define the higher Atiyah-Patodi-Singer index.  Let
\begin{enumerate}
\item $g$ denote a Riemannian metric on $W$ which is a product metric in a neighborhood of $\partial W$;
\item $\nabla_{F_{B_1}}$ a connection compatible with $g|_{\partial W}$;
\item $\nabla_{E_{B_2}}$ a connection which is compatible with $g$, $\nabla_{F_{B_1}}$, and the bundle isomorphism $\alpha$; 
\item $P$ a spectral section for the operator on the boundary (i.e., $D_{\partial W, F_{B_1}}$);
\end{enumerate}
With all this data fixed, results from \cite{LP} imply that there is a well-defined index 
$${\rm ind}(D^{P}_{W, E_{B_2}}) \in K_*(B_2)$$ 
As an element of $K_*(B_2)$, it depends on these choices (e.g., the metric, connections, and spectral section).  However, we will show that the image of this class under $r_*:K_*(B_2) \rightarrow K_{*+1}(C_{\phi})$ is independent of these choices.  
\par
To do so, a number of properties of the higher Atiyah-Patodi-Singer index are required.  These properties are that the higher Atiyah-Patodi-Singer index, spectral flow, and difference construction of spectral sections are each functorial.  The functorial properties of this index are discussed in \cite[Appendix C]{PS} while for spectral flow and the difference construction the reader can see \cite{Wahl}. 
\par
To state these properties precisely, additional notation is required. Recall that $\phi:B_1\rightarrow B_2$ is a unital $*$-homomorphism and $W$ is a compact ${\rm spin^c}$-manifold with boundary.  Further assume that $F^{\prime}_{B_1}$ is a $B_1$-bundle over all of $W$.  Let $P$ and $Q$ be spectral sections for $D_{\partial W, F^{\prime}_{B_1}|_{\partial W}}$.  The following three properties will be used
\begin{eqnarray}
\phi_*({\rm ind}^{B_1}(D^{P}_{W,F^{\prime}_{B_1}}))& =& {\rm ind}^{B_2}(D^{\phi_*(P)}_{W,F^{\prime}_{B_1} \otimes_{\phi} B_2}) \\
\phi_*({\rm sf}(D_{\partial W,F^{\prime}_{B_1}|_{\partial W},t};P,Q)) & = & {\rm sf}(D_{\partial W,F^{\prime}_{B_1}|_{\partial W}\otimes_{\phi}B_2,t};\phi_*(P),\phi_*(Q)) \label{sfFun} \\
\phi_*([P-Q]) & = & [\phi_*(P) - \phi_*(Q)] 
\end{eqnarray} 
where 
\begin{enumerate}
\item $D^{P}_{M,E}$ denotes the Dirac operator on $M$ twisted by $E$ with the boundary conditions associated to the spectral section $P$;
\item ${\rm ind}$ denotes the higher Atiyah-Patodi-Singer index;
\item ${\rm sf}(\: \cdot \:)$ denotes spectral flow (see \cite{Wahl} for further details); 
\item $[P-Q]\in K_*(B_1)$ denotes the difference class of $P$ and $Q$ (again further details can be found in \cite{MP} or \cite{Wahl});  
\end{enumerate}
\begin{define} \label{bouAnaIndDef}
Let $(W,(E_{B_2},F_{B_1},\alpha),f)$ be a cycle in $K_*(X;\phi)$ such that 
$${\rm ind}(D_{\partial W,F_{B_1}})=0\in K_{*+1}(B_1)$$  
Then, ${\rm ind}_{ana}(W,(E_{B_2},F_{B_1},\alpha),f):= r_*({\rm ind}(D^{P}_{W,E_{B_2}})) \in K_{*+1}(C_{\phi})$ where $P$ is any spectral section for $D_{\partial W, F_{B_1}}$ and $r_*: K_*(B_2) \rightarrow K_{*+1}(C_{\phi})$ is the map on $K$-theory induced from the $*$-homomorphism $r:SB_2 \rightarrow C_{\phi}$.
\end{define}
\begin{prop} \label{BCindexMapThm}
Let $(W,(E_{B_2},F_{B_1},\alpha),f)$ be a cycle in $K_*(X;\phi)$ and
assume that ${\rm ind}(D_{\partial W, F_{B_1}})=0$.  Then, the map ${\rm ind}_{ana}$ 
is well-defined as map on (isomorphism classes of) cycles. 
\end{prop}
\begin{proof}
A proof that the index map is well-defined at the level of cycles amounts to showing the right-hand side of the equation is independent of the choice of metric, connection, and spectral section used to define the higher Atiyah-Patodi-Singer index.  We begin with a special case; let
\begin{enumerate}
\item $\{g_t\}_{t\in [0,1]}$ be a one parameter family of Riemannian metrics on $W$;
\item $\nabla_{F_{B_1},t}$ be a one parameter family of connections on $F_{B_1}$ which is compatible with $g_t|_{\partial W}$;
\item $\nabla_{E_{B_2},t}$ be a one parameter family of connections on $E_{B_2}$ which is compatible with $g_t$ and with the family of connections $\nabla_{F_{B_1},t}$;
\item $\hat{P}_t$ be a one parameter family of spectral sections for $D_{\partial W, F_{B_1}}$.
\end{enumerate}
Set $P=\phi_*(\hat{P}_t)$.  By functorial properties of spectral sections and the fact that $E_{B_2}|_{\partial W} \cong F_{B_1}\otimes_{\phi}B_2$, both $P_0$ and $P_1$ are spectral sections for $D_{\partial W, E_{B_1}|_{\partial W}}$. Using this data, the following indices are well-defined: 
$${\rm ind}(D^{P_0}_{W,E_{B_2}}) \hbox{ and } {\rm ind}(D^{P_1}_{W,E_{B_2}})$$
Then, \cite[Proposition 8]{LP} implies that
$${\rm ind}(D^{P_0}_{W,E_{B_2}})-{\rm ind}(D^{P_1}_{W,E_{B_2}})={\rm sf}(\{D_{\partial W,(E_{B_2})|_{\partial W},t};P_0,P_1)\in K_*(B_2) $$ 
where ${\rm sf}(D_{\partial W,E|_{\partial W},t};P_0,P_1)$ is the spectral flow of the family of operators on the boundary (again see \cite{LP}).  Functorial properties of spectral flow (i.e., Equation \ref{sfFun}) imply that ${\rm sf}(D_{\partial W,E_{B_2},t};P_0,P_1)$ is in the image of $\phi_*$.  Exactness (i.e., $r_*\circ \phi_*=0$) leads to 
$$r_*({\rm ind}(D^{P_0}_{W,E_{B_2}}))-r_*({\rm ind}(D^{P_1}_{W,E_{B_2}}))=0 \in K_{*+1}(C_{\phi})$$
This completes the proof of the special case.
\par
The only difference in the general case is that we cannot assume that the spectral sections, $\hat{P}_0$ and $\hat{P}_1$, are joined via a one-parameter family.  However, there does exists a family of spectral section $\hat{Q}_t$.  As above, set $P_0=\phi_*(\hat{P}_0)$, $P_1=\phi_*(\hat{P}_1)$, and $Q_t=\phi_*(\hat{Q}_t)$.  Then, using \cite[Proposition 8 and Theorem 8]{LP}, we have
\begin{eqnarray*}
{\rm ind}(D^{P_0}_{W,E_{B_2}})-{\rm ind}(D^{P_1}_{W,E_{B_2}}) & = & {\rm ind}(D^{P_0}_{W,E_{B_2}})-{\rm ind}(D^{P_1}_{W,E_{B_2}}) \\
 & & -{\rm ind}(D^{Q_0}_{W,E_{B_2}})+{\rm ind}(D^{Q_1}_{W,E_{B_2}}) \\
 & & +{\rm sf}(\{D_{\partial W,(E_{B_2})|_{\partial W},t};Q_0,Q_1)  \\
& = & [Q_0-P_0]+ [P_1-Q_0] \\
& & +{\rm sf}(\{D_{\partial W,(E_{B_2})|_{\partial W},t};Q_0,Q_1)
\end{eqnarray*} 
Applying $r_*$ to this equation and using the functorial properties of the difference classes and spectral flow leads to
\begin{eqnarray*}
r_*({\rm ind}(D^{P_0}_{W,E_{B_2}}))-r_*({\rm ind}(D^{P_1}_{W,E_{B_2}})) & = & (r_*\circ \phi_*) ( [\hat{Q}_0-\hat{P}_0]+ [\hat{P}_1-\hat{Q}_0] \\ 
& & +{\rm sf}(\{D_{\partial W,(F_{B_1}),t};\hat{Q}_0,\hat{Q}_1))
\end{eqnarray*}
Exactness then implies the result.
\end{proof}
\begin{theorem}
If $\phi_* : K_*(B_1)\rightarrow K_*(B_2)$ is injective, then the analytic index map (see Definition \ref{bouAnaIndDef}) induces a well-defined map $ K_*(X;\phi) \rightarrow K_{*+1}(C_{\phi})$.  
\end{theorem}
\begin{proof}
The injective of $\phi_*$ implies that the conditions of Proposition \ref{BCindexMapThm} are satisfied for any cycle in $K_*(X;\phi)$. Thus the index map is well-defined at the level of cycles.  We need to show that the map respects the three relations.  \\
{\bf Disjoint union/direct sum:}  This follows from basic properties of the higher Atiyah-Patodi-Singer index.  \\
{\bf Bordism:} Let $(Z,W,(E^{\prime}_{B_2},F^{\prime}_{B_1},\alpha^{\prime}),g)$ be a bordism and $(W,(E_{B_2},F_{B_1},\alpha),f)$ denote its boundary.  Denote by $(M,V_{B_1},h)$ the $K_*(X;B_1)$-bordism obtained by restricting the given data to the ${\rm spin^c}$ manifold with boundary, $\partial W - {\rm int}(W)$.  Let $P$ and $Q$ be spectral sections for $D_{\partial W, F_{B_1}}$ and $D_{\partial M,V_{B_1}|_{\partial M}}$ respectively and $\tilde{P}$ and $\tilde{Q}$ denote the spectral sections (for $D_{\partial W, F_{B_1}\otimes_{\phi}B_2}$ and $D_{\partial M, V_{B_1}|_{\partial M}\otimes_{\phi}B_2}$ respectively) obtained via the $*$-homomorphism $\phi$. Using \cite[Theorem 8]{LP} and the functorial properties listed above, the indices on the various manifolds (we suppress the bundle data from the notation) involved are related via
\begin{eqnarray*}
{\rm ind}^{B_2}(D^{\tilde{P}}_W)+\phi_*({\rm ind}^{B_1}(D^{I-Q}_M) & = & {\rm ind}^{B_2}(D^{\tilde{P}}_W) + {\rm ind}^{B_2}(D^{I-\tilde{Q}}_M) \\
& = & {\rm ind}^{B_2}_{AS}(D_{W\cup M}) + [\tilde{P}-\tilde{Q}] \\
& = & {\rm ind}^{B_2}_{AS}(D_{W\cup M}) + \phi_*([P-Q]) 
\end{eqnarray*}
The fact that $r_* \circ \phi_*=0$ implies that
$$\mu_{ana}(W,(E_{B_2},F_{B_1},\alpha),f)=r_*({\rm ind}^{B_2}(D_W))=r_*({\rm ind}^{B_2}_{AS}(D_{W\cup M}))$$
Finally, the bordism invariance of the Mishchenko-Fomenko index and the fact that $W\cup M=\partial Z$ (the bundles respect this bordism) imply that the right-hand side of this equation vanishes.  This proves the required bordism invariance. \\ 
{\bf Vector bundle modification:}  Let $(W,(E_{B_2},F_{B_1},\alpha,f)$ denote a cycle and $V$ a ${\rm spin^c}$-vector bundle of even rank over $W$.  Since the higher Atiyah-Patodi-Singer index respects disjoint union, we may assume that $W$ is connected. Using Proposition \ref{APSandVBM}, we have that 
$${\rm ind}(D_W(A))={\rm ind}(D_{\hat{W}}(\hat{A})) \in K_*(B_2)$$
where we have used the notation of Proposition \ref{APSandVBM}. However, the definition of $\mu_{ana}$ is independent of the choice of spectral section (see Proposition \ref{BCindexMapThm}).  As such, 
\begin{eqnarray*}
\mu_{ana}(W,(E_{B_2},F_{B_1},\alpha,f)) & = & r_*({\rm ind}^{B_2}(D_W(A))) \\
& = & r_*({\rm ind}^{B_2}(D_{\hat{W}}(\hat{A}))) \\
& = & \mu_{ana}(W,(E_{B_2},F_{B_1},\alpha),f)^V)
\end{eqnarray*} 
\end{proof}
\begin{theorem} \label{indBouEqu}
Suppose that $\phi_*: K_*(B_1) \rightarrow K_*(B_2)$ is injective (so that analytic index is well-defined). Then the topological index and analytic index are equal.  In particular, in the case of $X=pt$, the analytic index is an isomorphism.
\end{theorem}
\begin{proof}
The second statement in the theorem follows from the first and the fact that the topological index is an isomorphism in the case of a point. 
To prove the first statement, note that both the topological index and analytic index factor through the map
$$K_*(X;\phi) \rightarrow K_*(pt;\phi)$$
defined at the level of cycles via $(W,\xi,f) \mapsto (W,\xi)$.  Thus, we need only show that they give the same isomorphism from $K_*(pt;\phi)$ to $K_{*+1}(C_{\phi})$. Using Theorem \ref{bockTypeSeq}, we have that  exactness and the injectivity of $\phi_*$ imply that the map $r:K_*(pt;B_2) \rightarrow K_{*+1}(pt;\phi)$ is onto.  This implies that given a cycle $(W,\xi) \in K_*(pt;\phi)$ there exists {\it closed} compact ${\rm spin^c}$-manifold $M$ and $\eta \in K^0(M;B_2)$ such that $r(M,\eta) \sim (W,\xi)$ (in the group $K_{*+1}(pt;\phi)$). The theorem now follows, since both the topological and analytic index of $(W,\xi)$ are equal to $\tilde{r}\circ {\rm ind}_{K_*(B_2)}(M,\eta)$ where ${\rm ind}_{K_*(B_2)}(M,\eta)$ denotes the Mishchenko-Fomenko index and $\tilde{r}:K_*(B_2) \rightarrow K_{*+1}(C_{\phi})$ is the map on $K$-theory induced from the natural $*$-homomorphism $SB_2 \rightarrow C_{\phi}$.
\end{proof}
\begin{remark}
Under the assumptions in the statement of the previous theorem, its proof implies that {\it any} index map $K_*(X;\phi) \rightarrow KK^*(\field{C},SC_{\phi})$ which agrees with the Mishchenko-Fomenko index on cycles without boundary is equal to the topological index map.  In particular, this statement holds (up to a factor of $-1$) for the index map discussed in \cite{DeeGeoRelKhom} for the special case when $\phi$ is the unital inclusion of the complex number into a ${\rm II}_1$-factor. Note that since the index map discussed in \cite{DeeGeoRelKhom} takes values in $\rz$, we must fix the isomorphism from $KK(\field{C},SC_{\phi})$ to $\rz$ to be the one compatible with isomorphism from $KK(\field{C},N)$ to $\field{R}$ defined via the trace of the ${\rm II}_1$-factor, $N$.  
\end{remark}
\vspace{0.25cm} 
\noindent 
{\bf Acknowledgments} \\
I thank Heath Emerson, Magnus Goffeng, Nigel Higson, Ralf Meyer and Thomas Schick for discussions.  This work was supported by NSERC through a postdoctoral fellowship held at Georg-August Universit${\rm \ddot{a}}$t, G${\rm \ddot{o}}$ttingen. I also thank the referee for a number of useful comments that have improved the exposition of the paper.

\vspace{0.25cm}
Email address: robin.deeley@gmail.com \vspace{0.25cm} \\
{ \footnotesize Department of Mathematics, University of Hawaii \\
2565 McCarthy Mall, Keller 401A \\
Honolulu HI 96822 USA}
\end{document}